\documentclass[12pt,a4paper]{amsart}
\usepackage[utf8]{inputenc}
\usepackage{amssymb}
\usepackage[plainpages=false,pdfpagelabels,colorlinks=true,citecolor=blue,hypertexnames=false]{hyperref}
\usepackage{url}
\usepackage{tikz}

\usepackage{todonotes}

\usepackage{xcolor}

\newtheorem{thm}{Theorem}[section]
\newtheorem{defn}[thm]{Definition}

\newtheorem{cor}[thm]{Corollary}
\newtheorem{ex}[thm]{Example}
\newtheorem{prop}[thm]{Proposition}

\newtheorem{lemma}[thm]{Lemma}

\newcommand{\bZ}{\mathbb{Z}}

\newcommand{\bN}{\mathbb{N}}
\newcommand{\bC}{\mathbb{C}}

\newcommand{\Dfn}[1]{\emph{\color{blue}#1}}

\newcommand{\oeis}[1]{\href{https://oeis.org/#1}{#1}}

\newcommand{\bt}{\mathcal{B}}
\newcommand{\ba}{\mathbf{a}}
\newcommand{\ie}{{\it i.e.}}
\newcommand{\GL}{\mathrm{GL}}
\newcommand{\SL}{\mathrm{SL}}

\title[On some sequences associated to invariant theory]{On some combinatorial sequences \\  associated to invariant theory}
\author[A.~Bostan]{Alin Bostan}
\address{Inria, Universit{\'e} Paris-Saclay, 1 rue Honor{\'e}
d'Estienne d'Orves, 91120 Palaiseau, France}
\thanks{A. Bostan was supported in part by {\href{https://specfun.inria.fr/chyzak/DeRerumNatura/}{DeRerumNatura}} ANR-19-CE40-0018.}
\email{Alin.Bostan@inria.fr}
\author[{J.~Tirrell}]{Jordan Tirrell}
\address{Department of
Mathematics and Computer Science, Washington College, USA}
\email{jtirrell2@washcoll.edu}
\author[{B.~W.~Westbury}]{Bruce W. Westbury}%
\email{Bruce.Westbury@gmail.com}%
\author[Y.~Zhang]{Yi Zhang}
\address{Department of Foundational  Mathematics, School of Science, Xi'an Jiaotong-Liverpool University,
 Suzhou, 215123, China}
\thanks{The work of Y. Zhang was supported by XJTLU Research Development Funding No.\ RDF-20-01-12, the NSFC Young Scientist Fund No.\ 12101506 and the Natural Science Foundation of the Jiangsu Higher Education Institutions of China No.\ 21KJB110032. }
\email{Yi.Zhang03@xjtlu.edu.cn}

\date{\today}

\begin{document}
	\begin{abstract}
We study the enumerative and analytic properties of 
some sequences constructed using tensor invariant theory.
The octant sequences are constructed 
from the exceptional Lie group~$G_2$ 
and the quadrant sequences from the special linear group~$SL(3)$. 
In each case we show that the corresponding sequences are related by binomial
transforms. The first three octant sequences and the first four quadrant sequences are listed in the On-Line Encyclopedia of Integer Sequences (OEIS).
These sequences all have interpretations as enumerating two-dimensional lattice walks but for
the octant sequences the boundary conditions are unconventional. These sequences are all P-recursive 
and we give the corresponding recurrence relations. 
In all cases the associated 
differential operators are of third order and have the remarkable property that they can be solved
to give closed formulae for the ordinary generating functions in terms of classical Gaussian hypergeometric functions.
Moreover, we show that the octant sequences and the quadrant sequences are related by the branching rules for the inclusion
of~$SL(3)$ in $G_2$.
	\end{abstract}
	\maketitle
	\tableofcontents

\section{Introduction} 
We study two families of sequences, each of them constructed using tensor invariant theory.
The theory is summarised in \S\ref{SEC:inv}.
The first family contains the \emph{octant sequences}. The first octant sequence is constructed from the seven-dimensional
fundamental representation of the exceptional simple algebraic group $G_2$.
This sequence was studied initially by Kuperberg in~\cite{Kuperberg1997}
and subsequently by Westbury in~\cite{Westbury2005}. We will refer to this sequence as $T_3$.
The second family contains the \emph{quadrant
sequences}. The first quadrant sequence is constructed from the direct sum of the three-dimensional vector
representation of $SL(3)$ and its dual. The sequences are defined as the dimension of the subspace of
invariant tensors in the tensor powers of the representation. In both cases we extend to a family of sequences
by adding copies of the trivial representation to the initial representation. Combinatorially, it means that
each family consists of the iterated binomial transforms of the first sequence.

The first three sequences of the octant family have entries in the
On-Line Encyclopedia of Integer Sequences (OEIS).
They are displayed in Figure~\ref{fig:seq}.
The sequence \oeis{A059710} is the first octant sequence $T_3$.
\begin{figure}[h]\label{fig:seq}
\begin{tabular}{c|rrrrrrrrrr}
$\ba$  
& 0 & 1 & 2 & 3 & 4 & 5 & 6 & 7 & 8 & 9 \\ \hline
\oeis{A059710} & 1 & 0 & 1 & 1 & 4 & 10 & 35 & 120 & 455 & 1792 \\
\oeis{A108307} & 1 & 1 & 2 & 5 & 15 & 51 & 191 & 772 & 3320 & 15032 \\
\oeis{A108304} & 1 & 2 & 5 & 15 & 52 & 202 & 859 & 3930 & 19095 & 97566
\end{tabular}
\caption{The first sequences of the family of octant sequences, 
their \href{https://oeis.org}{OEIS} tags
and their first ten terms.}
\end{figure}

To our knowledge, the sequence $T_3$ has not appeared previously in the combinatorics literature.
In \S\ref{sec:octant}
we prove the following combinatorial interpretations of this sequence.
\begin{thm} \label{THM:firstoctantcomb}
The sequence $T_3$ enumerates:
\begin{itemize}
\item hesitating tableaux of height 2, empty shape, and no singleton;
\item set partitions with no singleton and no enhanced 3-crossing;
\item sequences $(x_1,x_2,\dotsc ,x_n)$ such that $1\leqslant x_i< i$ with no weakly decreasing subsequence of length 3.
\end{itemize}
\end{thm}
These all follow from known combinatorial interpretations of the binomial transform of $T_3$;
they are corollaries to Theorem~\ref{THM:hes} in \S\ref{sec:octant}.

The starting point for Theorem~\ref{THM:firstoctantcomb} is the interpretation of the sequences in the first family
as lattice walks in the plane. This interpretation is used in \cite{Westbury2005}. These lattice walks have six
nonzero steps and one zero step corresponding to the seven weights of the representation. These walks are restricted
to the region $0\leqslant y\leqslant x$. There is an additional constraint which is unusual from the standard
theory of lattice walks, namely, that the zero step is forbidden on the line $x=y$. 

Let $E_3$ be the second octant sequence \oeis{A108307} in Figure~\ref{fig:seq}. Using this interpretation, we have the following theorem.

\begin{thm}\label{thm:bt} The sequence $E_3$ is the binomial transform of
the sequence $T_3$, \ie, for $n \geq 0$,
\[ E_3(n) = \sum_{i=0}^n \binom{n}{i}\,T_3(i). \]
\end{thm}
This result provides an unexpected connection between the invariant theory
of $G_2$ and the combinatorics of set partitions. 

In \S\ref{SUBSEC:algebra} we study the ordinary generating functions of these sequences. First we derive
the following recurrence relation for $T_3$. 

\begin{thm}\label{thm:rc}
The sequence $T_3$ is determined by the initial conditions
$ T_3\left( 0 \right)=1$, $T_3\left( 1 \right)=0$, $T_3\left( 2 \right)=1$ and the recurrence relation 
\begin{multline}
14\left( n+1 \right)  \left( n+2 \right) T_3 \left( n \right) +
 \left( n+2 \right) \left( 19n+75 \right)T_3 \left( n+1 \right) \\
+2 \left( n+2 \right) \left( 2n+11 \right) T_3\left( n+2 \right)
  - \left( n+8
 \right)  \left( n+9 \right) T_3\left( n+3 \right)=0
\end{multline}
satisfied for all $n \geq 0$.
\end{thm}

The recurrence relation in Theorem~\ref{thm:rc} was conjectured by Mihailovs~\cite[\S3]{Westbury2005}
and we give
three independent proofs in \S\ref{SUBSEC:algebra}. The corresponding linear differential equation, 
satisfied by the generating function $\mathcal T(t) = \sum_{n \geq 0} T_3(n) t^n$,
can
be solved to give a closed formula for $\mathcal T(t)$ in terms of the Gaussian hypergeometric function
${}_2F_1$. The differential equations and recurrence relations for the other two sequences in 
Figure~\ref{fig:seq} are derived in \cite{Bousquet-Melou2005} and closed formulae are given in the OEIS.

\medskip 
The first four sequences of the quadrant sequences also have entries in the OEIS; they are shown in Figure~\ref{fig:bx}.
The most well-known of these sequences is the third one, which enumerates \emph{Baxter permutations}~\cite{MR184217}.

These sequences also have interpretations as lattice walks. These walks are restricted to the quadrant $0\leqslant x,y$.
These are walks in six weights of the representation. This interpretation follows from the Pieri rule.
For $\GL(3)$, representations are indexed by partitions with three rows. Then applying the Pieri rule gives:
\begin{equation*}
    V_{[1]}\otimes V_\lambda \cong \oplus_\mu V_\mu
    \qquad
    V_{[1,1]}\otimes V_\lambda \cong \oplus_\nu V_\nu
\end{equation*}
where the first sum is over all partitions $\mu$ obtained by adding a single box to $\lambda$
and the second sum is over all partitions $\nu$ obtained by adding two boxes in different rows to $\lambda$.
Then the restriction to $\SL(3)$ is given by mapping the partition $[\lambda_1,\lambda_2,\lambda_3]$ to the
lattice point in the plane $(\lambda_1-\lambda_2,\lambda_2-\lambda_3)$.

The first sequence is studied in \cite{MR2681853}, where it is defined in terms of lattice walks.
The identification of the third sequence with the sequence enumerating Baxter permutations can be proved by several methods.
One method is to show that they both satisfy the same recurrence relation and initial conditions. Another method is to prove that they
both enumerate axis walks in the octant, using \cite{MR3738145}, \cite{Burrill2015}, \cite{MR3470878} or our derivation utilizing branching rules.
The fourth sequence was introduced by Marberg in \cite{Marberg2012}. It is defined as the number of noncrossing 2-coloured set partitions;
a lattice path interpretation is given in \cite{Marberg2012}, where a differential equation and a recurrence relation can also be found.
The second sequence does not have a published reference, but it is known to enumerate a family of unlabelled graphs called \emph{planar st-graphs}~\cite{Cranch2015}. 
We identify this sequence with the OEIS entry by giving a
combinatorial interpretation and by showing that the binomial transform is the third sequence in \S\ref{SEC:quadrant}. 
Moreover,  the second and fourth sequences are both given by axis walks
in the corresponding octant.

\begin{figure}[h]
    \centering
\begin{tabular}{c|rrrrrrrrr}
$\ba$ 
& 0 & 1 & 2 & 3 & 4 & 5 & 6 & 7 & 8 \\ \hline
\oeis{A151366} & 1 & 0 & 2 & 2 & 12 & 30 &130 & 462 & 1946 \\ 
\oeis{A236408} & 1 & 1 & 3 & 9 & 33 & 131 & 561 & 2535 & 11971 \\ 
\oeis{A001181} & 1 & 2 & 6 & 22 & 92 & 422 & 2074 & 10754 & 58202\\
\oeis{A216947} & 1 & 3 & 11 & 47 & 225 & 1173 &6529 & 38265 & 233795
\end{tabular}
\caption{The first four quadrant sequences, 
their \href{https://oeis.org}{OEIS} tags
and their first nine terms.}
\label{fig:bx}
\end{figure}

The ordinary generating functions for these sequences are studied in~\S\ref{SUBSEC:rec}. We give a recurrence relation 
for the $k$-th quadrant sequence,
which includes $k$ as a parameter. A closed formula for the generating function of the Baxter sequence is given in the OEIS without a reference. We give a different, but equivalent, closed formula.

These two families of (octant and quadrant) sequences are related since $SL(3)$ is a maximal subgroup of $G_2$.
Furthermore, the restriction of the seven-dimensional fundamental representation is the direct sum of the
two three-dimensional representations and one copy of the trivial representation. More generally, the restriction
of the representation for the $k$-th sequence in octant family is the $(k+1)$-st sequence in the quadrant family.
It means that these two sequences are related by the branching rules for the inclusion of $SL(3)$ in $G_2$.
It has been observed several times that the Baxter numbers enumerate walks which end on an axis
and bijections which establish this are given in \cite[Proposition~23]{Burrill2015}, 
\cite[Theorem~3]{MR3738145}, \cite[Theorem~1.2]{Yan2020}. In the last section we give an interpretation of this
result using representation theory. It gives a new proof that these sets are equinumerous but it is not
bijective. However, this proof does extend the result. First, the analogous result relates the $k$-th quadrant sequence
and the $k$-th octant sequence. It also follows combinatorially from the known result using the binomial
transforms. Second, our proof also extends this result by giving an analogous result for walks ending at a point
other than the origin.


This is an expanded version of \cite{BTWZ19}. The abbreviated version was used in \cite{Scherer21} to find
the radius of convergence of sequence \oeis{A060049} and proving \cite[Conjecture~8.2]{Kuperberg1997}.

\section{Invariant theory}\label{SEC:inv}
In this section we describe the way in which the representation theory of algebraic groups gives rise to combinatorially interesting
sequences. For the basic notions in the theory of linear algebraic groups, the reader is invited to consult the book~\cite{FuHa91}. 

Let $G$ be a reductive complex algebraic group, i.e., $G$ has a trivial unipotent radical.
Let $V_\lambda$ be the irreducible representation of $G$ with highest weight $\lambda$.
\begin{defn} \label{DEF:seq}
Let~$V$ be a (finite-dimensional) representation of~$G$ and $\lambda$ a dominant weight.
The {\emph sequence associated to $(G,V,\lambda)$},
denoted~$\ba_{V,\lambda}$, is the sequence whose $n$-th term is the multiplicity
of $V_\lambda$ in the tensor power $\otimes^nV$. 
\end{defn}

Our main interest is in the case $\lambda=0$. In this case $V_\lambda$ is the trivial representation
and the $n$-th term of the sequence $\ba_{V,\lambda}$ is the dimension of the subspace of
invariant tensors in $\otimes^nV$. The sequence $\ba_{V,0}$ is denoted by $\ba_{V}$.

\begin{ex}\label{ex:A1}
The simplest example is to take $G=SL(2)$ and $V$ to be the two-dimensional defining representation.
The odd terms of the sequence $\ba_V$ are zero and the even terms are the Catalan numbers~$C_n = \frac{1}{n+1} \binom{2n}{n}$.
\end{ex}

It is very classical that the generating function for the sequence of Catalan numbers is \emph{algebraic}.
In our context, this can be seen as a special case
of the following general result.

\begin{prop} 
Take $G=SL(2)$ and $V$ to be any representation. Then the generating
function of the sequence $\ba_V$ is algebraic.
\end{prop}
\begin{proof}
The character $\chi(V)$ of $V$ is a Laurent 
polynomial in $x$ invariant under $x\leftrightarrow x^{-1}$. Then the $n$-th term of the sequence~$\ba_V$
is the constant term of the Laurent polynomial $\chi(V)^n(1-x^2)$. 
Equivalently, 
the generating function of the sequence $a_V$
is the diagonal of the bivariate rational function
$\frac{x-x^{-1}}{1-tx \,\chi(V)}$. 
The result now follows by applying the Pólya-Furstenberg theorem~\cite{Polya22,Furstenberg67}, see also~\cite[Theorem~6.3.3]{Stanley1999}.
\end{proof}

Next we discuss the case where $V$ is a fundamental representation of a classical simple Lie algebra
of rank two. These sequences have all been studied and are known to be \emph{P-recursive}, or \emph{holonomic}: they satisfy linear recurrences with polynomial coefficients (in the index~$n$).

\begin{ex}
Take $G=SL(3)$ and $V$ to be the three-dimensional defining representation.
In the sequence $\ba_V$ every third term is non-zero and the other terms are zero.
Taking every third term gives the sequence of three-dimensional Catalan numbers (\oeis{A005789}), whose $n$-th term is given by $2 \, (3n)!/(n! \, (n+1)! \, (n+2)!)$
and matches the number
of standard tableaux of rectangular shape $(n,n,n)$.

This can be seen as follows. Let $V$ be the vector representation of~$GL(n)$. Then Schur-Weyl duality gives the decomposition
of $\otimes^r V$ as $\oplus_\lambda S_\lambda\otimes V_\lambda$ where $\lambda$ is a partition of size $r$
with at most $n$ non-zero parts and $S_\lambda$ is an irreducible representation of the symmetric group
on~$r$ points. There are several constructions (due to Young) of $S_\lambda$,  which take the vector
space with basis the set of standard tableaux of shape $\lambda$ and give an action of the symmetric group.
Now restrict from $GL(n)$ to $SL(n)$. The restriction of the representation $V_\lambda$ is trivial if~$\lambda$ has rectangular shape and is non-trivial otherwise.
\end{ex}

\begin{ex}
Take $G=Sp(4)$ of type $C_2$ and $V$ to be the four-dimensional defining representation.
Under the isomorphism~$Sp(4)\cong Spin(5)$, we see that $V$ is the spin representation of the spin group $Spin(5)$.
The odd terms in the sequence $\ba_V$ are zero and the even terms are those of the sequence 
$C_n C_{n+2} - C_{n+1}^2$ (\oeis{A005700}).
This sequence counts 3-noncrossing perfect matchings on $2n$ points 
and also  nested pairs of Dyck paths of length $2n$.
\end{ex}

The theory of Kashiwara crystals (see~\cite[Ch.~2]{Malle2011}, or \cite[Ch.~4]{Hong2002})
gives combinatorial interpretations of these sequences.
There is a crystal associated to each representation. The crystal associated to
the irreducible representation $V_\lambda$ is denoted by $C_\lambda$. These are the
connected crystals. Crystals have a tensor product with the property that the
crystal associated to the tensor product of representations is the tensor product
of the crystals associated to the representations. The trivial crystal is $C_0$
and has one element. Moreover, the crystal associated to the direct sum of representations
is the disjoint union of the crystals associated to the representations.

Let $C$ be a crystal. The elements of $\otimes^nC$ are words of length $n$ in the elements of $C$.
The analogue of the subspace of invariant tensors is the subset of invariant words in $\otimes^nC$.

The crystal version of Definition~\ref{DEF:seq} is:
\begin{defn} \label{DEF:crystal}
Let~$C$ be a (finite) crystal of~$G$ and $\lambda$ a dominant weight.
The {\emph sequence associated to $(G,C,\lambda)$}, denoted~$\ba_{C,\lambda}$,
is the sequence whose $n$-th term is the multiplicity
of the crystal $C_\lambda$ in the tensor power $\otimes^nC$. 
\end{defn}


\begin{prop}\label{prop:ic}
If $C$ is the crystal of the representation $V$ then $\ba_{C,\lambda}=\ba_{V,\lambda}$
for all dominant weights $\lambda$.
\end{prop}

In this paper we only apply this to the sum of the two fundamental representations of $\SL(3)$
and to the seven-dimensional fundamental representation of $G_2$. For the fundamental representations
of $SL(3)$ this follows from classical results by interpreting a standard or semistandard tableau
as a sequence of partitions. For the $G_2$ case this result has been used in 
\cite{Kuperberg1997} and \cite{Westbury2005}. We give an outline of the general case which is
an application of the tensor product rule for crystals. This is \cite[Theorem~2]{Grabiner1993}.

\begin{proof} A tensor product of highest weight representations has a 
canonical vector space decomposition as
\begin{equation*}
    V_\lambda\otimes V_\mu \cong \oplus_\nu A^{\nu}_{\lambda,\mu}\otimes V_\nu.
\end{equation*}
Similarly the tensor product of crystals has a set decomposition
\begin{equation*}
    C_\lambda\otimes C_\mu \cong \coprod_\nu B^{\nu}_{\lambda,\mu}\otimes C_\nu.
\end{equation*}
The decomposition rule of \cite{Littelmann1994} asserts that $\dim A^{\nu}_{\lambda,\mu} = |B^{\nu}_{\lambda,\mu}|$.

Then the $n$-th terms of the two sequences are equal. This is proved by induction on $n$ using
the above for the inductive step.
\end{proof}

\subsection{Binomial transform}\label{sec:bt}
In this subsection we recall some basic properties of the binomial transform.
Identities associated to the binomial transform are given for instance in \cite{Boyadzhiev2018}.

The \emph{binomial transform operator} is denoted by $\bt$. This is a linear operator
acting on sequences.
Given the sequence $\ba$ with $n$-th term $a(n)$, the \Dfn{binomial transform
of $\ba$} is the sequence, denoted $\bt \ba$, whose $n$-th term is
\[
\sum_{i=0}^n \binom{n}{i}a(i).
\]

Binomial transforms arise naturally for sequences associated to the representations of
reductive complex algebraic groups.
\begin{lemma} \label{LEM:binomialrep} 
Assume that $\ba_{V,\lambda}$ is the sequence associated to $(G, V,\lambda)$ as specified in
Definition~\ref{DEF:seq}. Then $\ba_{V\oplus \bC,\lambda} = \bt \ba_{V,\lambda}$. 
\end{lemma}
\begin{proof}
Expanding the tensor product gives
\begin{equation*}
    \otimes^n(V\oplus\bC) \cong \oplus_{i=0}^n \bC^{\binom ni}\otimes (\otimes^iV) .
\end{equation*}
This can be proved by induction on $n$ as in the binomial theorem.

Now comparing invariant tensors on both sides gives the result.
\end{proof}

The binomial transform also arises naturally for lattice walks restricted to a domain.
\begin{lemma} \label{LEM:binomialwalks}
Assume that a sequence $\ba$ enumerates walks in a lattice, confined to a
domain $D$, using a set of steps~$S$. Then $\bt \ba$ is given by adding a new
step corresponding to the zero element of the lattice; that is, $S$ is
replaced by the disjoint union $S\coprod \{0\}$ without changing the domain.
\end{lemma}
\begin{proof}
A path of length $n$ with steps $S\coprod \{0\}$ has $i$ steps $S$ and $j$ steps~0 with
$i+j=n$. Taking the steps in $S$ and ignoring the 0 steps gives a path of length~$i$ in~$S$.
Each path of length~$i$ in~$S$ appears $\binom nj$ times corresponding to the partitions of
a set of size~$n$ into two disjoint subsets of size~$i$ and~$j$.
\end{proof}

The following is a corollary to both Lemma~\ref{LEM:binomialrep} and Lemma~\ref{LEM:binomialwalks}.
\begin{cor} \label{COR:binomialcrystal} 
Assume that $\ba_{C,\lambda}$ is the sequence associated to $(G, C,\lambda)$ as specified in
Definition~\ref{DEF:crystal}.
Let $\ast$ be the trivial crystal. Then 
$$\ba_{C\coprod \ast} = \bt \ba_{C,\lambda}.$$
\end{cor}

We can also consider iterations of the binomial transform.
\begin{lemma}\label{lem:bts} 
Given the sequence $\ba$ with $n$-th term $a(n)$. 
For $k\in\bZ$, the $k$-th binomial transform of $\ba$ is given by
\[
(\bt^k \ba)(n) = \sum_{i=0}^n k^{n-i}\binom{n}{i}a(i) \ \text{ for each } \ n \in \bN.
\]
\end{lemma}

The binomial transform can also be regarded as an operator on the generating
function of a sequence. Let $G(t) = \sum_{n = 0}^\infty a(n) t^n$ be the
generating function of the sequence $\ba$. We denote the generating function
of $\bt^k \ba$ by $\bt^kG$. Then we have the following immediate lemma.
\begin{lemma}\label{lem:btg}
For $k\in\bZ$, the $k$-th binomial transform of $G(t)$ is
\[ (\bt^kG)(t) = \frac 1{1-k\;t} G\left(\frac {t}{1-k\;t}\right). \]
\end{lemma}

\section{Octant sequences}\label{sec:octant}
In this section we consider a sequence associated to the representation theory of the exceptional
simple algebraic group $G_2$ by the construction in Definition~\ref{DEF:seq}.
Then we relate this sequence and its binomial transforms to known sequences of hesitating
tableaux of height 2 and vacillating tableaux of height 2. These were introduced in \cite{Chen2007}.

The algebraic group $G_2$ has dimension 14 and rank 2. It can be constructed as the automorphism group
of the (complex) octonions. By construction, the octonions are then an eight-dimensional representation of $G_2$.
The unit spans a one-dimensional invariant subspace and the imaginary octonions are an invariant complementary
subspace. This seven-dimensional representation of $G_2$ is a fundamental representation.

\subsection{Root systems}
The root system of $G_2$ 
(see \cite[Lecture~22]{FuHa91})
is shown in Figure~\ref{fig:G2_roots} with the fundamental chamber shaded. The two simple roots are $\alpha$ and $\beta$ with $\alpha$ short and $\beta$ long. The two fundamental weights are $\lambda$ and $\theta$. The highest weight representation with highest weight $\lambda$ is the seven-dimensional representation we are interested in. The highest weight representation with highest weight $\theta$ is the adjoint representation, so $\theta$ is the highest root. 

These are related by
\[ \lambda = 2\,\alpha+\beta, \qquad \theta = 3\,\alpha+2\,\beta.
\]

The element $\rho$ is defined in terms of the roots as half the sum of the positive roots. The six positive roots are
\[ \alpha, \beta, \alpha+\beta, 2\,\alpha+\beta, 3\,\alpha+\beta, 3\,\alpha+2\,\beta. \]
This gives $\rho = 5\,\alpha+3\,\beta$.

The element $\rho$ is also defined in terms of the weights as the sum of the fundamental weights. This gives $\rho=\lambda+\theta$.
\begin{figure}
    \centering
\begin{tikzpicture}
    \foreach\ang in {30,60,...,360}{
\draw[-,green!80!black,thin] (0,0) -- (\ang:3.5cm);
    }
    \draw[->,red,thick] (0,0) -- (0,3);
    \node[anchor=south west] at (0,3) {$\beta$};
    \draw[->,red,thick] (0,0) -- (300:2cm);
    \node[anchor=south west] at (300:2cm) {$\alpha$};
    \fill[blue!20] (0,0) -- (3.5,0) -- (30:3.5cm) -- cycle;
    \draw[->,black,thick] (0,0) -- (2,0);
    \node[anchor=north west] at (2,0) {$\lambda$};
    \draw[->,black,thick] (0,0) -- (30:3cm);
    \node[anchor=south west] at (30:3cm) {$\theta$};
\end{tikzpicture}
    \caption{$G_2$ root system}
    \label{fig:G2_roots}
\end{figure}
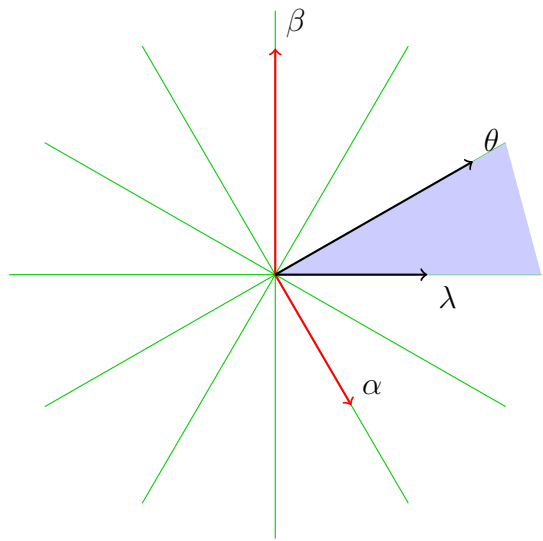

\subsection{Lattice walks}\label{SUBSEC:latticewalks}
In this subsection, we give a combinatorial proof of Theorem~\ref{thm:bt}.
Let $G$ be $G_2$, let $V$ be the seven-dimensional fundamental representation, and $C$ be
the associated crystal. We now apply the general theory of \S\ref{SEC:inv}
to this case.

\begin{figure}
    \centering
\begin{tikzpicture}[scale=2]
    \fill[blue!20] (0,0) -- (2,0) arc(0:30:2) -- cycle;
    \draw[green] (1,0) -- (60:1) -- (120:1) -- (-1,0) -- (240:1) -- (300:1) -- cycle;
    \draw[->,blue] (0,0) -- (1,0) node {(1,0)};
    \draw[->,blue] (0,0) -- (60:1) node {(-1,1)};
    \draw[->,blue] (0,0) -- (120:1) node {(-2,1)};
    \draw[->,blue] (0,0) -- (-1,0) node {(-1,0)};
    \draw[->,blue] (0,0) -- (240:1) node {(1,-1)};
    \draw[->,blue] (0,0) -- (300:1) node {(2,-1)};
    \draw[blue] (0,0) node {(0,0)};
\end{tikzpicture}
    \caption{Steps in weight lattice}
    \label{fig:stepswt}
\end{figure}
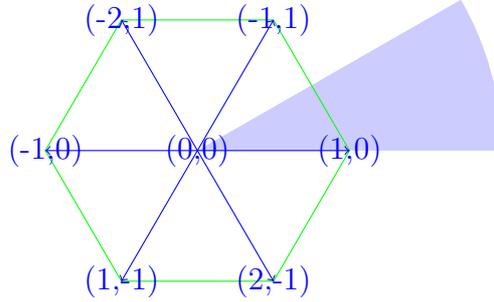

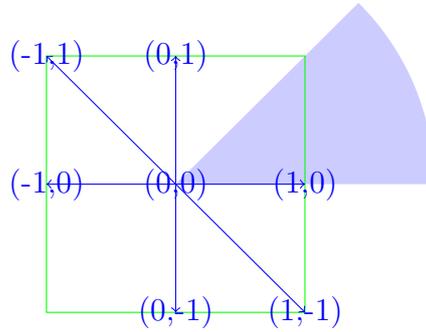
\begin{figure}
    \centering
\begin{tikzpicture}[scale=1.7]
    \fill[blue!20] (0,0) -- (2,0) arc(0:45:2) -- cycle;
    \draw[green] (1,1) -- (1,-1) -- (-1,-1) -- (-1,1) -- cycle;
    \draw[->,blue] (0,0) -- (1,0) node {(1,0)};
    \draw[->,blue] (0,0) -- (0,1) node {(0,1)};
    \draw[->,blue] (0,0) -- (-1,1) node {(-1,1)};
    \draw[->,blue] (0,0) -- (-1,0) node {(-1,0)};
    \draw[->,blue] (0,0) -- (0,-1) node {(0,-1)};
    \draw[->,blue] (0,0) -- (1,-1) node {(1,-1)};
    \draw[blue] (0,0) node {(0,0)};
\end{tikzpicture}
    \caption{Steps in octant}\label{fig:stepswk}
\end{figure}

The highest weight words in tensor powers of $C$ correspond to lattice walks 
in the weight lattice of $G_2$ restricted to the dominant chamber.
These lattice walks are studied by Kuperberg~\cite{Kuperberg1997} and by Westbury~\cite{Westbury2005}.
Since $V$ has dimension~7, $C$ has cardinality~7, and the seven steps are shown in
Figure~\ref{fig:stepswt}, where the coordinates are with respect to the basis of fundamental weights. 
One can make it more explicit by taking the intersection of the Euclidean lattice in three-dimensional space with the plane $x+y+z=0$. 
Then, we can take the short simple root $\alpha = (0,1,-1)$, the long simple root $\beta = (1,-2,1)$,
 the fundamental weight $\lambda = (1,0,-1)$,
and the fundamental weight $\theta = (2,-1,-1)$.

The highest weight words correspond with these steps, restricted to the dominant
chamber and with the extra condition that the step $(0,0)$ is not permitted at
a weight $(0,k)$ for $k\geqslant 0$.

A \Dfn{hesitating tableau} of semilength $n$ is a walk in the Young lattice with $n$ steps.
Each step is one of the following pairs of moves on the Young lattice:
\begin{itemize}
    \item do nothing, add a cell
    \item remove a cell, do nothing
    \item add a cell, remove a cell
\end{itemize}
The \Dfn{shape} is the final partition.

A hesitating tableau of \Dfn{height} $h$
is a hesitating tableau such that every partition in the sequence has height at most $h$.

There is a lattice walk interpretation of hesitating tableaux of height~2 given by 
Bousquet-M\'elou and Xin in~\cite{Bousquet-Melou2005}.

A hesitating tableau of height 2 can be interpreted as a walk in $\bZ^2$ by identifying
partitions with at most two nonzero rows with the set
\[ \{(x,y)\in\bZ^2 | x\geqslant y\geqslant 0 \} . \]
There are 8 steps which are shown in Figure~\ref{fig:stepswk}.
There are eight steps since there are:
\begin{itemize}
    \item two ways to do nothing then add a cell,
    \item two ways to remove a cell then do nothing
    \item four ways to remove a cell then add a cell
\end{itemize}
Two of the four ways to remove a cell then add a cell give
the step~$(0,0)$, namely add a cell on the first row then remove this cell and add a cell on the second row then remove this cell. It is always allowed to add and then remove a cell on the first row. The step
which adds and removes a cell on the second row is not permitted
on the line $x=y$.

\begin{proof}[Proof of Theorem~\ref{thm:bt}]

We compare the two descriptions of the walks. The steps of the walks for the
sequence $T_3$ are shown in Figure~\ref{fig:stepswt} and the steps of the
walks for the sequence $E_3$ are shown in Figure~\ref{fig:stepswk}.

In order to compare these, we first make the change of coordinates $(x,y) \to
(x + y, x)$, which maps the steps from  Figure~\ref{fig:stepswt} to those of Figure~\ref{fig:stepswk}. This identifies the six non-zero steps, as well as the two
domains.

By Lemma~\ref{LEM:binomialwalks}, it remains to compare the zero steps. There
is one zero step in the~$T_3$ case and two in the $E_3$ case. In the $E_3$
case we have the zero step which adds and then removes a cell on the second
row. The boundary condition is that this step is not allowed on the line
$x=y$. After the change of coordinates, this is the same boundary condition as
the zero step in the $T_3$ case. The second zero step in the $E_3$ case adds
and then removes a cell in the first row. This is always allowed. 

\end{proof} 

\begin{thm}\label{THM:hes} For each $n\geqslant 0$ and $r,s\geqslant 0$ there is a correspondence between highest weight words
of weight $(r,s)$ and length $n$ in the crystal $C\coprod\ast$ and hesitating tableaux of height 2,
semilength $n$ and shape $(r+s,s)$.
\end{thm}

\begin{proof} The proof is the observation that the two descriptions of the lattice walks agree under a simple
change of coordinates. The point $(r,s)$ in the weight lattice is dominant if $r,s\geqslant 0$. This point
corresponds to the two part partition $(r+s,s)$.

It is straightforward to check that under this change of coordinates, the eight steps correspond, the two domains
correspond, and the extra condition on a step of weight $(0,0)$ also corresponds.
\end{proof}

There is a variation of Theorem~\ref{THM:hes} for vacillating tableaux.
A \Dfn{vacillating tableau} of semilength $n$ is an excursion in the Young lattice with~$n$ steps.
Each step is one of the following pairs of moves on the Young lattice:
\begin{itemize}
    \item do nothing twice
    \item do nothing, add a cell
    \item remove a cell, do nothing
    \item remove a cell, add a cell
\end{itemize}
A vacillating tableau of \Dfn{height} $h$
is a vacillating tableau such that every partition in the sequence has height at most $h$.

There is a lattice walk interpretation of vacillating tableaux of height~2 also given in \cite{Bousquet-Melou2005}.
There are nine steps since there are six steps which change the shape and three steps which leave the shape unaltered.

Vacillating tableaux were introduced in \cite{Chen2007} in the context of set partitions.
A \Dfn{set partition} will mean a partition of the ordered set
$$[n]=\{1,2,\dotsc,n\}$$ 
 into non-empty pairwise disjoint subsets.
The \Dfn{standard representation} of a set partition is a graph
with vertex set $[n]$. The edges are arcs connecting pairs of
elements in a block which are adjacent in the numerical order
on the block.

A \Dfn{singleton} in a set partition is a block with one element.

A \Dfn{$k$-crossing} in a set partition is a $k$-subset of arcs,
$(i_1,j_1),\dotsc,(i_k,j_k)$ such that
\[ i_1<i_2\dotsb <i_k<j_1<j_2\dotsb <j_k
\]

An \Dfn{enhanced $k$-crossing} in a set partition is a $k$-subset of arcs,
$(i_1,j_1),\dotsc,(i_k,j_k)$ such that
\[ i_1<i_2\dotsb <i_k\leqslant j_1<j_2\dotsb <j_k
\]

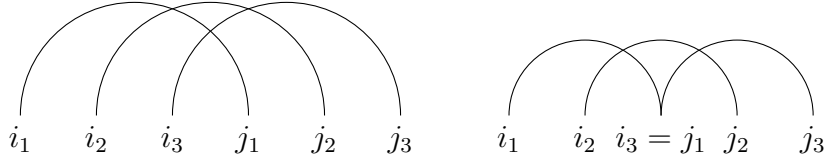
\begin{figure}
    \centering
\begin{tikzpicture}
\draw (3,0) node[below] {$j_1$} arc(0:180:1.5) node[below] {$i_1$};
\draw (4,0) node[below] {$j_2$} arc(0:180:1.5) node[below] {$i_2$};
\draw (5,0) node[below] {$j_3$} arc(0:180:1.5) node[below] {$i_3$};
\end{tikzpicture}\qquad
\begin{tikzpicture}
\draw (2,0) node[below] {$i_3=j_1$} arc(0:180:1) node[below] {$i_1$};
\draw (3,0) node[below] {$j_2$} arc(0:180:1) node[below] {$i_2$};
\draw (4,0) node[below] {$j_3$} arc(0:180:1);
\end{tikzpicture}
    \caption{A 3-crossing and an enhanced 3-crossing}
    \label{fig:3cr}
\end{figure}

\begin{thm}\label{THM:vac} For each $n\geqslant 0$ and $r,s\geqslant 0$ there is a correspondence between highest weight words
of weight $(r,s)$ and length $n$ in the crystal $C\coprod\ast\coprod\ast$ and vacillating tableaux of height 2,
semilength $n$ and shape $(r+s,s)$.
\end{thm}

\begin{proof}
The proof is essentially the same as the proof of Theorem~\ref{THM:hes}.
\end{proof}

The main innovation in \cite{Chen2007} is the construction of
a bijection between set partitions of $[n]$ and vacillating
tableaux of semilength $n$ and empty shape. The set partition has a $k$-crossing
if and only if some partition in the vacillating tableau has length
at least $k$. In particular, for $k=3$, this is a bijection between
3-noncrossing set partitions and vacillating tableaux of height 2
and empty shape.

There is a correspondence between hesitating tableaux of semilength $n$,
empty shape and height $h$ and set partitions of $n$ with no enhanced $h+1$-crossing.

This is proved by Lin in~\cite{MR3779614} and by Marberg in~\cite[Proposition~5.8]{Marberg2012}.
In \cite{gil19} the sequence \oeis{A108307} is called $NC_{0,3}$, the sequence \oeis{A108304} is called $NC_{1,3}$;
the binomial transform relation between them is given in~\cite{MR3779614, gil19}. 

A corollary is that there is a correspondence between invariant words of length $n$ in $C$
and set partitions of $[n]$ with no singleton and no enhanced 3-crossing.

There is another combinatorial interpretation based on \cite{MR3168516}
and \cite{MR3779614}.
\begin{defn}
An \Dfn{inversion sequence} of length $n$ is a sequence \\
$(x_1,x_2,\dotsc ,x_n)$ such that $1\leqslant x_i\leqslant i$.
\end{defn}

A \Dfn{singleton} in the inversion sequence  $(x_1,x_2,\dotsc ,x_n)$
is an $i\in [n]$ such that $x_i=i$.

\begin{prop} For $n\geqslant 0$ there is a bijection between
set partitions of $[n]$ with no enhanced 3-crossing and no
singletons and  inversion sequences of length $n$ with no weakly decreasing subsequence of length~3 and no singletons.
\end{prop}

\begin{proof} There is a bijection in \cite{MR3168516}
between set partitions of $[n]$ with no enhanced 3-crossing
and  inversion sequences of length $n$ with no weakly decreasing subsequence of length 3. 
This bijection preserves singletons.
\end{proof}

\subsection{Algebra}\label{SUBSEC:algebra}
In this subsection, we give three proofs of Theorem~\ref{thm:rc}, with
different flavors. The first two proofs utilize Theorem~\ref{thm:bt} and a result
of Bousquet-M\'elou and Xin~\cite{Bousquet-Melou2005} on partitions that avoid
3-crossings. The last one is a direct proof which relies on the connection with $G_2$ walks.

The first proof is based on Theorem~\ref{thm:bt}, on~Proposition~2
in~\cite{Bousquet-Melou2005} and on the method of creative
telescoping~\cite{Zeilberger1991, KoutschanThesis} for the summation of
(bivariate) holonomic sequences.

\begin{prop}\label{main-thm-H}\cite[Proposition 2]{Bousquet-Melou2005}
The number $E_3(n)$  of partitions of $[n]$ having no enhanced
$3$-crossing is given by $ E_3\left( 0 \right)=E_3\left( 1 \right)
=1,$ and for
$n\ge 0$,
\begin{multline}
8\left( n+3 \right)  \left( n+1 \right) E_3 \left( n \right) +
 \left( 7{n}^{2}+53n+88 \right) E_3 \left( n+1 \right) \\
 - \left( n+8
 \right)  \left( n+7 \right) E_3\left( n+2 \right)=0.
\end{multline}
Equivalently, the associated generating function $\mathcal E(t)=\sum_{n\ge 0} E_3(n)t^n$
satisfies
\begin{multline*}\label{e-dfinite-H}
{t}^{2} \left(1+ t \right)  \left(1- 8t \right) {\frac
{d^{2}}{d{t} ^{2}}}\mathcal E ( t )
+2t  \left( 6-23t-20{t}^{2} \right) {\frac {d}{dt}}\mathcal E ( t ) \\
+6 \left(5-7t -4{t}^{2}\right) \mathcal E ( t ) -30=0.
\end{multline*}
\end{prop}

\begin{proof}[First Proof of Theorem~\ref{thm:rc}]
By Theorem~\ref{thm:bt}, we have 
\[ T_3(n) = \sum_{k = 0}^n (-1)^{n - k} \binom nk E_3(k). \]
Set $f(n, k) = (-1)^{n - k} \binom nk E_3(k)$. By
Proposition~\ref{main-thm-H}, and by the closure properties for holonomic
functions, it follows that $f(n, k)$ is holonomic. Thus, we can apply Chyzak's
algorithm~\cite{KoutschanThesis} for creative telescoping to derive a
recurrence relation for~$T_3$. In particular, using Koutschan's
Mathematica package {\tt HolonomicFunctions.m}~\cite{Christoph2010} that
implements Chyzak's algorithm, we find exactly the recurrence equation in
Theorem~\ref{thm:rc}.

\end{proof}
\noindent The detailed calculation involved in the above proof can be found
in~\cite{ElectronicYZ}.

\medskip 
The second proof is also based on Proposition~\ref{main-thm-H}
and on Theorem~\ref{thm:bt}, namely on the relation between the generating
functions of $T_3(n)$ and of $E_3(n)$ implied by Theorem~\ref{thm:bt}.

\begin{proof}[Second Proof of Theorem~\ref{thm:rc}] Let $\mathcal T(t) = \sum_{n \geq 0} T_3(n) t^n$. 
By Theorem~\ref{thm:bt} and Lemma~\ref{lem:btg},
\[ \mathcal T(t) = \frac{1}{1 + t} \cdot \mathcal E\left( \frac{t}{1 + t} 
\right). \]
We know from Proposition~\ref{main-thm-H} a differential equation for
$\mathcal E\left( t \right)$. By (univariate) closure properties of D-finite
functions, we deduce a differential equation for $\mathcal T(t)$, and convert
it into a linear recurrence for $T_3(n)$, which is exactly the recurrence in
Theorem~\ref{thm:rc}. \end{proof}

Let $G$ be a reductive complex algebraic group and $V$ be a representation of~$G$. 
If $G$ is connected then the sequence~$\ba_V$ associated to $(G, V)$ can sometimes be
written as an algebraic residue. The general principle is discussed by Gessel and Zeilberger in~\cite{MR1092920}
and in detail by Grabiner and Magyar in~\cite{Grabiner1993}.

Let $K$ be the character of $V$ and let $\Delta$ be given by
\[ \Delta = \sum_{w\in W} \varepsilon(w)\left[ w(\rho)-\rho \right] . \]
Here $W$ is the Weyl group, $\varepsilon\colon W\to \{\pm 1\}$ is the sign character and $\rho$
is half the sum of the positive roots. These are both elements of the group ring of the root lattice of $G$.
If we choose a basis of the root lattice then these become Laurent polynomials where the number of
indeterminates is the rank of $G$. 

For the representations that appear in this paper, the $n$-th term of the sequence~$\ba_V$ is the constant term in
the Laurent polynomial $\Delta\,K^n$ and the generating function is the algebraic residue of the
rational function $\Delta/(x y-t x y K)$.

\begin{ex}
In Example~\ref{ex:A1}, the Laurent polynomials are
\[ \Delta = (1-x^{-2}), \qquad K = (x-x^{-1}).\]
Then the constant term of the Laurent polynomial 
\[ (1-x^{-2})(x-x^{-1})^{2n} \]
is
\[ \binom{2n}{n}-\binom{2n}{n+2}, \]
which is a well known expression for the $n$-th Catalan number~$C_n$. This can be proved directly
using the reflection principle.
\end{ex}

It is a general result that an algebraic residue is D-finite, see \cite{MR929767}. This shows that the generating functions
of these sequences are D-finite and hence the sequences are P-recursive. 

The third proof of Theorem~\ref{thm:rc} relies on~$G_2$ walks and the method
in~\cite{MR3588720}.

The following definition is given in \cite{MR1265145} in the paragraph
following Conjecture~3.3.

\begin{defn}\label{def:W}
The $n$-th term $T_3(n)$ is the constant term of the Laurent polynomial
$\Delta\,K^n$, where
\[ K = 1 + x + y + x\,y + x^{-1} + y^{-1} + (xy)^{-1}, \]
and $\Delta$ is the Laurent polynomial
\begin{multline*}
\Delta = x^{-2}y^{-3}(x^2y^3 - xy^3 + x^{-1}y^2 - x^{-2}y +x^{-3}y^{-1} - x^{-3}y^{-2}\\ + x^{-2}y^{-3} - x^{-1}y^{-3} + xy^{-2} - x^2y^{-1} + x^3y - x^3y^2).
\end{multline*}
\end{defn}
\begin{proof}[Third Proof of Theorem~\ref{thm:rc}] 
Let $\mathcal T(t) = \sum_{n \geq 0} T_3(n) t^n$. By Definition~\ref{def:W},
$\mathcal T(t)$ is the constant coefficient $[x^0 y^0]$ of $\Delta/(1-tK)$. In
other words, $\mathcal T(t)$ is equal to the algebraic residue of
$\Delta/(xy-txyK)$, and thus it is D-finite~\cite{MR929767} and is annihilated by a
linear differential operator that cancels the contour integral of
$\Delta/(xy-txyK)$ over a cycle. Using the integration algorithm for multivariate
rational functions in~\cite{Dwork} we compute the following operator of
order~6, that cancels $\mathcal T(t)$:
\begin{align*}
L_6 = {t}^{5} \left( t+1 \right)  \left( 7\,t -1\right)  \left( 2\,t+1 \right) ^{2}\partial^{6}+ \\
3\,{t}^{4} \left( 2\,t+1 \right) 
 \left( 168\,{t}^{3}+211\,{t}^{2}+40\,t-11 \right) \partial^{5}+ \\
6\,{t}^{3} \left( 2100\,{t}^{4}+3475\,{t}^{3}+1616\,{t}^{2}+
79\,t-61 \right) \partial^{4}+ \\
6\,{t}^{2} \left( 11200\,{t}^{4}+17400\,{t}^{3}+7556\,{t}^{2}+268\,t-273 \right) \partial^{3}+ \\
36\,t \left( 4200\,{t}^{4}+6100\,{t}^{3}+2442\,{t}^{2}+54\,t-77 \right) \partial^{2} +\\
36\, \left( 3360\,{t}^{4}+4540\,{t}^{3}
+1646\,{t}^{2}+16\,t-35 \right) \partial+\\
20160\,{t}^{3}+25200\,{t}^{2}+8064\,t,
\end{align*}
where $\partial = \frac{\partial}{\partial t}$ denotes the derivation operator
with respect to~$t$.

The operator $L_6$ factors as $L_6 = Q L_3$, where
\[
Q
=
 \left( 2\,t+1 \right) {t}^{3}{\partial}^{3}+ \left( 24\,t+13 \right) {t}^{2}{\partial}^{2}+6\, \left( 12\,t+7 \right) t{\partial}+48\,t+30,
\]
and
\begin{multline} \label{EQ:L3}
L_3 = {t}^{2} \left( 2\,t+1 \right)  \left( 7\,t -1\right)  \left( t+1 \right) {\partial}^{3}+ \\ 2\,t \left( t+1 \right)  \left( 63\,{t
}^{2}+22\,t-7 \right) {\partial}^{2}+ 	 \\ 
 \left( 252\,{t}^{3}+338\,{t}^{2}+36\,t-42 \right) {\partial}+28\,t \left( 3\,t+4 \right). 
\end{multline}
This shows that $f(t) := L_3(\mathcal T(t))$ is a solution of the differential
operator $Q$. Hence, by denoting $f(t) = \sum_{n \geq 0} f_n t^n$, one deduces
that for all $n\geq 0$ we have \[2 \left( n+2 \right) f_n + \left( n+6 \right)
f_{n+1} = 0. \] One the other hand, from $\mathcal T(t) = 1+t^2 + O(t^3)$, it
follows that~$f_0=0$, therefore $f(t) = 0$, in other words $\mathcal T(t)$ is
also a solution of~$L_3$. From there, deducing the recurrence relation of
Theorem~\ref{thm:rc} is immediate. 

\end{proof}

\subsection{Closed formulae} \label{SUBSEC:closedformulae}
In this subsection, we give two closed formulae for the generating function $\mathcal T(t)$  of  sequence~$T_3$ (\oeis{A059710})  in terms of 
classical Gaussian hypergeometric functions ${}_2F_1$.

The operator $L_3$ in~\eqref{EQ:L3} factors as $L_3 = L_2 L_1$, where
\begin{align*}
L_2 =
{t}^{2} \left( 2\,t+1 \right)  \left( 7\,t - 1\right)  \left( t+1 \right) {{\partial}}^{2}+{\frac {t (t+1) P_1 }P} {\partial}+
{\frac {
P_2}{P^2}},
\end{align*}
and
\begin{align*}
L_1 = 
\partial - \frac{\frac{\rm d}{dt} (P/t^5)}{P/t^5}
\end{align*}
with
\[
P_1 = 
3136\,{t}^{6}+7560\,{t}^{5}+5744\,{t}^{4}+1592\,{t}^{3}-90\,{t}^{2}-131\,t-9,
\]
\begin{align*}
P_2=
131712\,{t}^{11}+719712\,{t}^{10}+1626800\,{t}^{9}+2014088\,{t}^{8}+1498264\,{t}^{7}+\\ 665620\,{t}^{6}+ 146508\,{t}^{5}-4560\,{t}^
{4}-11138\,{t}^{3}-2663\,{t}^{2}-244\,t-7
\end{align*}
and
\[
P=
{28\,{t}^{4}+66\,{t}^{3}+46\,{t}^{2}+15\,t+1}.
\]
Clearly, the operator $L_1$ has a basis of solutions formed of $\{ P/t^5\}$.
Let $\{ f_1, f_2 \}$ be a basis of solutions of $L_2$. Then one can show that the solution $\mathcal T(t)$  of $L_3$ is 
\begin{equation*} 
\frac{P}{t^5}
\cdot
\int (C_1 f_1 + C_2 f_2) \frac{t^5}{P} dt,
\end{equation*}
for some constants $C_1$ and $C_2$ to be determined.

Using algorithms for solving second order differential equations, as described
in~\cite{MR3588720}, one deduces that $f_1$ and $f_2$ have hypergeometric
expressions. After identifying the constants $C_1$ and $C_2$, we derive the following explicit formula for $\mathcal T(t)$.
Define the polynomials:
\begin{gather*}
S = (7\,t-1)\,(t+1)\,(2\,t+1)^2, \\
U = (11\,t-1)\,(46\,t^3-78\,t^2+15\,t-1),
\end{gather*}
and
\begin{multline*}
V =  11870\,t^7-6934\,t^6-13371\,t^5+1115\,t^4 \\
+1112\,t^3-300\,t^2+29\,t-1,
\end{multline*}
as well as the rational function
\begin{equation*}
\phi = \frac{27\,(t+1)\,t^2}{(1-t)^3}.
\end{equation*}
Then the generating function $\mathcal T(t)$ is equal to
\begin{equation*}
\setlength\arraycolsep{1pt}
\frac P{30 \, t^5}\int\frac{S}{P^2\,(t-1)^2}\left[
U\, {}_2F_1\left(\begin{matrix}\frac{1}{3}& &\frac{2}{3}\\&1&
\end{matrix};\phi\right) +
\frac{V}{(1-t)^3}\,{}_2 F_1\left(\begin{matrix}\frac{4}{3}& &\frac{5}{3}\\&2&
\end{matrix};\phi\right)
\right]dt.
\end{equation*}
Note that the above indefinite integrals have no constant term. This expression can be further simplified by introducing
the additional rational functions
\[ R_1 = 
{\frac { \left( t+1 \right) ^{2} \left( 214\,{t}^{3}+45\,{t}^{2}+60\,t+5 \right) }{t-1}}
\]
and
\[
R_2 = 
6\,{\frac { {t}^{2} \left( t+1 \right) ^{2} \left( 101\,{t}^{2}+74\,t+5 \right) }{ \left( t-1 \right) ^{2}}}.
\]
Then the simpler expression for $\mathcal T(t)$ is
\begin{equation*}
\mathcal T(t) = 
\setlength\arraycolsep{1pt}
\frac {1}{30 \, t^5} \left[
R_1\cdot {}_2F_1\left(\begin{matrix}\frac{1}{3}& &\frac{2}{3}\\&2&
\end{matrix};\phi\right) +
R_2 \cdot {}_2 F_1\left(\begin{matrix}\frac{2}{3}& &\frac{4}{3}\\&3&
\end{matrix};\phi\right) + 5 \, P
\right].
\end{equation*}

Following the approach in~\cite[\S3.3]{MR3588720}, one can obtain an
alternative expression for $\mathcal T(t)$ by using an approach with a more geometric flavor. 
The key point is that
the denominator of $W/(xy-txyK)$ is a family of \emph{elliptic} curves, thus
integrating $W/(xy-txyK)$ over a small torus amounts to computing the periods
of the two forms (of the first and second kind). Working out the details, this
approach yields an expression for $\mathcal T(t)$ in terms of the Weierstrass invariant
\[ g_2 = \left( t-1 \right)  \left( 25\,{t}^{3} + 21\,{t}^{2} + 3\,t - 1 \right) \]
and of the $j$-invariant
\[
J = 
{\frac { \left( t-1 \right)^3  \left( 25\,{t}^{3} + 21\,{t}^{2} + 3\,t - 1 \right) ^{3}}{{t}^{6} \left( 1-7\,t \right)  \left( 2
\,t+1 \right) ^{2} \left( t+1 \right) ^{3}}}
\]
of our family of curves. As in~\cite{MR3588720}, we introduce the expression 
\[
\setlength\arraycolsep{1pt}
H(t)  = 
\frac {1}{{g_2}^{1/4}} \cdot 
{}_2 F_1\left(\begin{matrix}\frac{1}{12}& &\frac{5}{12}\\& 1 &
\end{matrix};\frac{1728}{J}\right).
\]

Then the generating function $\mathcal T(t)$ is equal to
\begin{align*}
	\frac{P}{6 \, t^5} + 
\frac{	\left( 7\,t - 1\right) 
	 \left( 2\,t+1 \right)  \left( t+1 \right)}{360 \, t^5} 
	   \Big(  \left( 155\,{t}^{2}+182\,t+59 \right) \left( 11\,t+1 \right)  {\it H} \left( t \right)+ \\
 \left( 341\,{t}^{3}+507\,{t}^
	{2}+231\,t+1 \right) \left( 5\,t+1 \right)  H'(t)
\Big).
\end{align*}	

Moreover, one can use the methods in~\cite{MR3588720} to deduce that the generating
function $\mathcal T(t)$ is a transcendental power series and that the asymptotic
behavior of its $n$-th coefficient $T_3(n)$ is \[ T_3(n) \sim C\cdot \frac{7^n}{n^7}, \quad \text{where} \;
C = \frac{4117715}{864}{\frac {\sqrt {3}}{\pi}} \approx 2627.6.\] 

\section{Quadrant sequences} \label{SEC:quadrant}
In this section we consider a  family of sequences associated to the representation theory of the
simple algebraic group $SL(3)$ by the construction Definition~\ref{DEF:seq}.
The algebraic group $SL(3)$ has dimension 8 and rank 2. Let $V$ be the three-dimensional
vector representation and $V^*$ be the dual representation. These are the two fundamental
representations of $SL(3)$.

\begin{defn}\label{defn:quad}
For $k\geqslant 0$, the quadrant sequences $\mathcal{S}_k$ are the
sequence associated to $(SL(3), V\oplus V^*\oplus k\;\bC)$.
\end{defn}
By Lemma~\ref{LEM:binomialrep}, those sequences are also
related by binomial transforms.

By the general theory in Proposition~\ref{prop:ic} these sequences enumerate
walks in the weight lattice of $SL(3)$. Equivalently, these are walks in the
positive quadrant which is our reason for calling them \emph{quadrant sequences}.

The $n$-th term of the sequence $\mathcal{S}_0$ enumerates $A_2$-webs with $n$ boundary points,
see \cite{Kuperberg1997}. If we use the representation theory of $GL(3)$ (instead of $SL(3)$)
then we can translate this to a combinatorial interpretation in terms of rectangular semistandard
tableaux with three rows.

Let $\mathrm{SST}(\lambda;\alpha)$ be the set of semistandard tableaux of shape $\lambda$
and weight $\alpha$. Here $\lambda$ is a partition, $\alpha$ is a composition and 
$|\lambda|=|\alpha|$.

\begin{prop}\label{prop:sst}
The first three quadrant sequences have the following interpretations in terms of rectangular semistandard tableaux.
\begin{enumerate}
\item The $n$-th term of the sequence $\mathcal{S}_0$ is the number of rectangular semistandard tableaux
with three rows whose weight, $\alpha$,
is a composition of length $n$ such that $\alpha_i\in\{1,2\}$ for $1\leqslant i\leqslant n$.

\item The $n$-th term of the sequence $\mathcal{S}_1$ is the number of rectangular semistandard tableaux with three rows whose weight, $\alpha$,
is a composition of length $n$ such that $\alpha_i\in\{0,1,2\}$ for $1\leqslant i\leqslant n$.

\item The $n$-th term of the sequence $\mathcal{S}_1$ is the number of rectangular semistandard tableaux with three rows whose weight, $\alpha$,
is a composition of length $n$ such that $\alpha_i\in\{1,2,3\}$ for $1\leqslant i\leqslant n$.

\item The $n$-th term of the sequence $\mathcal{S}_2$ is the number of rectangular semistandard tableaux with three rows whose weight, $\alpha$,
is a composition of length $n$ such that $\alpha_i\in\{0,1,2,3\}$ for $1\leqslant i\leqslant n$.
\end{enumerate}
\end{prop}

\begin{proof} There are easy bijections with the respective walks. A semistandard tableaux can be regarded
as a sequence of partitions where the $i$-th partition is the shape given by taking the cells whose entry is at most $i$.
For a semistandard tableau with (at most) three rows this gives a sequence of vectors $(\lambda_1,\lambda_2,\lambda_3)$
such that $\lambda_1\geqslant \lambda_2\geqslant \lambda_3\geqslant 0$. Each such vector maps to the point
$(\lambda_1-\lambda_2,\lambda_2-\lambda_3)$. This constructs a map from semistandard tableaux with three rows to walks
in the two-dimensional lattice which stay in the non-negative quadrant. This map gives all bijections.
\end{proof}

\begin{prop}\label{prop:enum} The first three quadrant sequences have the following expressions. 

\begin{enumerate}
\item The $n$-th term of the sequence $\mathcal{S}_0$ is
\begin{equation*}
    \sum_{m\geqslant 0}\sum_{\substack{a,b,c\geqslant 0\\a+b=n\\ a+2b=3m}}%
   \binom{n}{a,b} \left| \mathrm{SST}(3^m;1^a2^b)\right|
\end{equation*}

\item The $n$-th term of the sequence $\mathcal{S}_1$ is
\begin{equation*}
    \sum_{m\geqslant 0}\sum_{\substack{a,b,c\geqslant 0\\a+b+c=n\\ b+2c=3m}}%
   \binom{n}{a,b,c} \left| \mathrm{SST}(3^m;0^a1^b2^c)\right|
\end{equation*}

\item The $n$-th term of the sequence $\mathcal{S}_1$ is
\begin{equation*}
    \sum_{m\geqslant 0}\sum_{\substack{a,b,c\geqslant 0\\a+b+c=n\\ a+2b+3c=3m}}%
   \binom{n}{a,b,c} \left| \mathrm{SST}(3^m;1^a2^b3^c)\right|
\end{equation*}

\item The $n$-th term of the sequence $\mathcal{S}_2$ is
\begin{equation*}
    \sum_{m\geqslant 0}\sum_{\substack{a,b,c,d\geqslant 0\\a+b+c=n\\ b+2c+3d=3m}}%
   \binom{n}{a,b,c,d} \left| \mathrm{SST}(3^m;1^b2^c3^d)\right|
\end{equation*}
\end{enumerate}
\end{prop}
\begin{proof}
These follow directly from Proposition~\ref{prop:sst} together with the result that
$|\mathrm{SST}(\lambda;\alpha)|=|\mathrm{SST}(\lambda;\alpha')|$ if $\alpha'$ is a reordering of $\alpha$.
\end{proof}

Note that it is clear from these expressions that $\mathcal{S}_1$ is the binomial transform of $\mathcal{S}_0$
and that $\mathcal{S}_2$ is the binomial transform of $\mathcal{S}_1$.
This uses the observation that $\left|\mathrm{SST}(3^m;1^a2^b)\right|=\left|\mathrm{SST}(3^{m+c};1^a2^b3^c)\right|$
for all $a,b,c,m\geqslant 0$ such that $a+2b=m$.

\begin{thm} \label{THM:quadrantid} 
The quadrant sequences $\mathcal{S}_0, \mathcal{S}_1, \mathcal{S}_2, \mathcal{S}_3$
are identical to the sequences in the second family specified in Figure~\ref{fig:bx}.
\end{thm}

\begin{proof}
   The sequence $\mathcal{S}_0$ appears as sequence \oeis{A151366}.
This sequence is defined in \cite{MR2681853} using lattice walks.

The sequence $\mathcal{S}_1$ appears as sequence \oeis{A236408}.
There is no published reference for this sequence and we discuss this sequence in
\S~\ref{sec:maps} below.

The sequence $\mathcal{S}_2$ appears as sequence \oeis{A001181}.
This is the sequence of Baxter permutations. These two sequences can be seen
to be the same by noting that Proposition~\ref{prop:enum} together with
the known formula for the number of rectangular partitions agrees with the
formula in \cite{MR491652} for the number of Baxter permutations.
Alternatively one can check that both sequences satisfy the same recurrence
relation and initial conditions. A bijective proof is given by noting that there is an easy bijection using the combinatorial interpretation in Proposition~\ref{prop:sst}
and \cite[Theorem~2]{Dulucq1996}. Alternative proofs are given in \cite[Proposition~23]{Burrill2015},
and \cite[Theorem~3]{MR3738145}.

The sequence $\mathcal{S}_3$ appears as sequence \oeis{A216947}.
This sequence is defined in \cite{Marberg2012} using 2-colored noncrossing set partitions.
This paper also gives an easy bijection with the lattice paths.
\end{proof}

Definition~\ref{DEF:seq} generalises to give sequences
$\ba_{V\oplus V^*\oplus k\;\bC,\lambda}$. Here $\lambda=(r,s)$ for $r,s\geqslant 0$
and the sequence enumerates lattice walks which start at $(0,0)$ and end at $(r,s)$.

\subsection{Root systems} 
The root system of $SL(3)$ is shown in Figure~\ref{fig:A2_roots} with the fundamental chamber shaded. The two simple roots are $\alpha$ and $\alpha'$. The two fundamental weights are $\lambda$ and $\lambda'$. The fundamental representations are the three-dimensional vector representation and its dual.

The element $\rho$ is
\[ \rho = \alpha+\alpha'=\lambda+\lambda'. \]
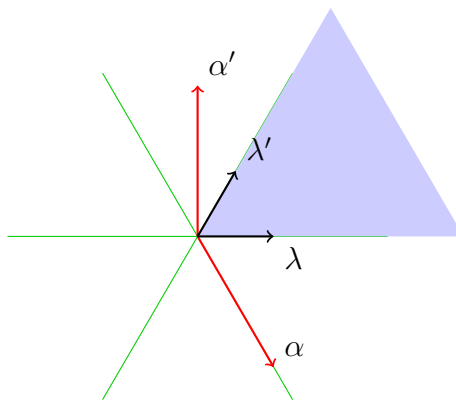
\begin{figure}
    \centering
\begin{tikzpicture}
    \foreach\ang in {60,120,...,360}{
\draw[-,green!80!black,thin] (0,0) -- (\ang:2.5cm);}
    \draw[->,red,thick] (0,0) -- (0,2);
    \node[anchor=south west] at (0,2) {$\alpha'$};
    \draw[->,red,thick] (0,0) -- (300:2cm);
    \node[anchor=south west] at (300:2cm) {$\alpha$};
    \fill[blue!20] (0,0) -- (3.5,0) -- (60:3.5cm) -- cycle;
    \draw[->,black,thick] (0,0) -- (1,0);
    \node[anchor=north west] at (1,0) {$\lambda$};
    \draw[->,black,thick] (0,0) -- (60:1cm);
    \node[anchor=south west] at (60:1cm) {$\lambda'$};
\end{tikzpicture}
    \caption{$A_2$ root system}
    \label{fig:A2_roots}
\end{figure}

All four entries in the OEIS give a recurrence relation for each quadrant sequence.
These recurrence relations all agree with, or are consequences of, the 
recurrence relations in \S\ref{SUBSEC:rec}. If the recurrence relation in the OEIS
has a proof then we can check that the initial terms agree and so we have an
alternative proof that the two sequences agree. If the recurrence relation in the OEIS
is conjectural then our results prove that these recurrence relations are 
satisfied by the sequence.

As a consequence of the identification of $\mathcal{S}_2$ with the Baxter sequence
we can give a formula for the terms of $\mathcal{S}_k$. This formula is obtained
by substituting the formula (1) in  \cite{MR491652} for the terms in the Baxter sequence
into the formula in Lemma~\ref{lem:bts} for iterated binomial transforms.

Other known results that are consequences are:
A bijective proof of the following is given in
\cite[Proposition~23]{Burrill2015}
and \cite[Theorem~3]{MR3738145},
and in a recent paper by Yan~\cite{Yan2018}. 
\begin{prop}\label{prop:hes}
The set of hesitating tableaux of length $n$ which end with a single row partition is equinumerous with the set of Baxter permutations on $[n]$.
\end{prop}

A bijective proof of the following can be found in~\cite[Corollary~14]{MR3738145}.
\begin{prop}\label{prop:vac}
The set of vacillating tableaux of length $n$ which end with a single row partition is equinumerous with the set of 2-coloured noncrossing set partitions of $[n]$.
\end{prop}

%
%

\subsection{Increasing maps}\label{sec:maps}
In this subsection we give a combinatorial interpretation of the sequence $\mathcal{S}_1$.
This depends on the closely related combinatorial interpretations of the Baxter sequence
in \cite{Ackerman2006} (mosaic floor plans), \cite{MR2734180} (plane bipolar orientations), 
and \cite{Law2012} (diagonal rectangulations). All three of these interpretations are
equivalent to oriented st-maps.

A \Dfn{map} is a connected graph embedded in the plane with no edge-crossings, considered up to isotopy.
The vertices and edges of the map are those of the graph. The faces of the map are the connected components of
the complement of the graph in the plane.
The \Dfn{outer face} is unbounded, the inner faces are bounded.

An \Dfn{oriented map} is a map together with an orientation of each edge. An oriented map is an \Dfn{st-map}
if it has a unique source, $s$, and a unique sink, $t$, both on the outer face.

An \Dfn{increasing map} is an oriented map together with an increasing labelling of the vertices.
If the map has $r$ vertices then the labelling, $\ell$, is a bijection between the vertices and the set
$\{1,2,\dotsc,r\}$. A labelling is increasing if whenever there is a directed path from vertex $v$ to vertex $w$
then $\ell(v)<\ell(w)$. Let $O$ be an increasing st-map. Then $\ell(s)=1$ and $\ell(t)=r$.

A \Dfn{simple increasing st-map} is a increasing st-map whose underlying graph has no multiple edges.

\begin{prop} \label{PROP:stmaps}
Let $a(n)$ be the number of increasing st-maps with~$n$ edges and let $a_S(n)$ be the number of simple increasing st-maps with~$n$ edges.
Then
\begin{equation*}
    a(n+1) = \sum_{k=0}^n \binom{n}{k} \,a_S(k+1)
\end{equation*}
\end{prop}

\begin{proof} There is a bijection between simple increasing st-maps with each edge labelled by
a positive integer and increasing st-maps since we can replace each edge labelled by $m$ with $m$
copies of the same edge. This is a bijection between increasing st-maps with $n$ edges and
pairs consisting of a simple increasing st-map with $k$ edges and a labelling of the edges by
positive integers whose sum is $n$. The result follows from the observation that the
number of sequences of positive integers of length $k$ with sum $n$ is
$\binom{n+1}{k+1}$.
\end{proof}

One corollary of Proposition~\ref{PROP:stmaps} is that the conjectured recurrence relation for \oeis{A236408} is correct.

Plane bipolar orientations arise in the unpublished notes by James Cranch,
\href{http://jdc41.user.srcf.net/research/pasting/Pasting.pdf}{Pasting Diagrams},
as the morphisms in a strict monoidal category.

First we construct the category. The set of objects is $\bN$ so we have an object $[n]$ for each $n\in \bN$.
Let $O$ be an oriented st-map.
One of the oriented paths going from $s$ to $t$ has the outer face on its right: this path is the \Dfn{right border} of $O$,
and its length is the \Dfn{right outer degree} of $O$. The \Dfn{left border} and \Dfn{left outer degree} is defined similarly.
An oriented st-map, $O$, is a morphism $[l]\to [r]$
where $l$ is the left outer degree of $O$ and $r$ is the right outer degree.
Composition of $O$ and $O'$ is only defined if the right outer degree of $O$ is the left outer degree of $O'$.
In this case, the composite $O\circ O'$ is the oriented st-map given by identifying the right border of $O$
with the left border of $O'$. The identity morphism of $[n]$ is a path with $n$ edges.

This category has a tensor product. On objects this is just $[n]\otimes [m] = [n+m]$  so the monoid of objects is $\bN$.
The tensor product $O\otimes O'$ is always defined and is the plane bipolar orientation
given by identifying the sink of $O$ with the source of $O'$. It is straightforward to check that this tensor
product is strictly associative and that these are compatible in the sense that they satisfy the interchange condition
$(O\circ O')\otimes (P\circ P') = (O\otimes P)\circ (O'\otimes P')$. Hence this is a strict monoidal category, $\mathcal{P}$.
Simple plane bipolar orientations are the morphisms of a strict monoidal subcategory, $\mathcal{P}_S$.

The monoidal category $\mathcal{P}$ is the free strict monoidal category whose monoid of objects is $\bN$
with a morphism $\alpha(n,m)\colon [n]\to [m]$ for $n,m>0$.
The monoidal subcategory $\mathcal{P}_S$ is the free strict monoidal category whose monoid of objects is $\bN$
with a morphism $\alpha(n,m)\colon [n]\to [m]$ for $n,m>0$ with $\alpha(1,1)$ omitted.
The morphism $\alpha(n,m)$ consists of a single inner face with $n$ edges on the left border and $m$ edges on the right border.

\subsection{Branching rules}
In this subsection we use the branching rules for the inclusion of
$SL(3)$ in $G_2$ to relate the two families of sequences.

The interpretation using invariant theory gives a connection between the octant
sequences and the quadrant sequences. This uses the fact that $SL(3)$ is a subgroup
of $G_2$, in fact, a maximal subgroup. Then the restriction of the representation
$V\oplus k\bC$ from $G_2$ to $SL(3)$ is the representation $V\oplus V^*\oplus (k+1)\bC$.
This implies that each quadrant sequence $\ba_{V\oplus k\;\bC,\lambda}$ can be written
as a linear combination of the octant sequences $\ba_{V\oplus k\;\bC,\mu}$ with
coefficients which are independent of $k$. These coefficients are the branching rules
for the inclusion of $SL(3)$ in $G_2$. A combinatorial description of the branching rules for the inclusion
\[ SL(3)\to G_2\] is given in \cite{King1978} and \cite{Richter2012}.

A special case of this is the result that
walks in the quadrant which end at $(0,0)$ correspond to walks in the octant which end
on the $x$-axis,  see \cite{Burrill2015}, \cite{MR3738145}, \cite{Burrill2015}, \cite{MR3470878},
\cite{Yan2020}.


The generating function for the branching rules for the inclusion of the maximal subgroup \[ SL(3)\to G_2\] are given by 
Gaskell and  Sharp in \cite[(2.3)]{MR638077}. Let $V(r,s)$ be a highest weight representation of $G_2$ with highest weight $(r,s)$ and let $U(p,q)$ be a highest weight representation of $SL(3)$ with highest weight $(p,q)$. Denote the multiplicity of $U(p,q)$ in the restriction of $V(r,s)$ to $SL(3)$ by $m^{(r,s)}_{(p,q)}$.

\begin{prop}
The generating function
\[
\sum_{r,s,p,q\geqslant 0} m^{(r,s)}_{(p,q)}
x^p\,y^q\,X^r\,Y^s
\]
is the rational function
\begin{equation}\label{eq:gf} \frac%
{(1-X)^{-1}-xyY(1-xyY)^{-1}}%
{(1-x\,X)(1-y\,X)(1-x\,Y)(1-y\,Y)}
\end{equation}
\end{prop}

\begin{thm} Let $V$ be a representation of $G_2$ and $V\downarrow$ the restriction to $SL(3)$. Then the number of axis-walks of length $n$ for $V$ is the number of excursions for $V\downarrow$.
\end{thm}

\begin{proof}
Putting $x=0$, $y=0$ in the generating function~\eqref{eq:gf}  gives
\[
\sum_{r,s,p,q\geqslant 0} m^{(r,s)}_{(0,0)}
X^r\,Y^s.
\]
Putting $x=0$, $y=0$ in the rational function gives $(1-X)^{-1}$. This shows that
\[
m^{(r,s)}_{(0,0)} = \begin{cases}
1 & \text{ if $s=0$}, \\
0 & \text{ otherwise}.
\end{cases}
\]
\end{proof}

\begin{ex}
First we extend each octant sequence to a sequence of functions
on the dominant weights of $G_2$, or, equivalently, on the 
octant $y\geqslant 0$, $x\geqslant y$. It follows from
Lemma~\ref{lem:bts} that these can be written as
polynomials in $k$.

Then the octant sequences are  given by
\[
\begin{tabular}{|c}
 1 \\
\hline
\end{tabular}\quad
\begin{tabular}{|cc}
 $k$ & 1 \\
\hline
\end{tabular}\quad
\begin{tabular}{|lll}
& 1 \\
$1+k^2$ & $1+2\,k$ & 1 \\
\hline
\end{tabular}
\]
\[
\begin{tabular}{|llll}
& $2+2\,k$ & 2 \\
$1+3\,k+k^3$ & $4+3\,k+3\,k^2$ & $3+3\,k$ & 1 \\
\hline
\end{tabular}
\]
Then summing the entries on the bottom row gives the sequence
\[
1,1+k,3+2\,k+k^2,9+9\,k+3\,k^2+k^3.
\]
\end{ex}

This proves Proposition~\ref{prop:hes} and 
Proposition~\ref{prop:vac}.


\subsection{Recurrence equations} \label{SUBSEC:rec}
In this subsection we give recurrence relations for the sequences $\mathcal{S}_k$.
A recurrence relation for $\mathcal{S}_3$ is given by Marberg in \cite[\S~4]{Marberg2012}.
\begin{lemma}\label{lemma:quad2}
Let $\mathcal{G}_k$ be the generating function of $\mathcal{S}_k$, where $k \geq 0$.
Then $\mathcal{G}_k$ is the constant coefficient $[x^0 y^0]$ of $W/(1-tK)$, where
\begin{equation} \label{EQ:K}
K = k+x+y+{x}^{-1}+{y}^{-1}+{\frac {x}{y}}+{\frac {y}{x}}
\end{equation}
and
\begin{equation}\label{EQ:W}
W = 1-{\frac {{x}^{2}}{y}}+{x}^{3}-{x}^{2}{y}^{2}+{y}^{3}-{\frac {{y}^{2}}{x}}.
\end{equation}
\end{lemma}

Let $C_2(n)$ be the $n$-th term of the sequence~\oeis{A216947} in the second
family. Marberg~\cite[Theorem 1.7]{Marberg2012} showed that $C_2(n)$ is the
constant term $[x^0 y^0]$ of $W \tilde{K}^n$, where $\tilde{K} = K|_{k = 3}$,
the Laurent polynomials $K$ and $W$ are specified in~\eqref{EQ:K}
and~\eqref{EQ:W}.
By Lemma~\ref{lemma:quad2}, the sequence $\mathcal{S}_3$ is identical to  the sequence~\oeis{A216947}. 
Therefore, we have:

\begin{prop}~\cite[Theorem 1.7]{Marberg2012} \label{LEM:marberg}
The $n$-th term $C_2(n)$ of the sequence~$\mathcal{S}_3$ is given by $C_2(0) = 1, C_2(1) = 3$ and for $n \geq 0$: 
\begin{multline*}
(n + 5) (n + 6) \cdot C_2(n +2)  - 2 (5 n^2 + 36 n + 61) \cdot C_2(n + 1) \\
 + 9 (n+1) (n +4) \cdot C_2(n) = 0,
\end{multline*}
Equivalently, the associated generating function $\mathcal{G}_3(t) = \sum_{n \ge 0} C_2(n) t^n$ satisfies 
\begin{multline*}
72 \mathcal{G}_3(t)+4 (-61 + 117 t) {\frac {d}{dt}}\mathcal{G}_3(t) +2 (15 - 184 t + 234 t^2) {\frac {d^2}{dt^2}}\mathcal{G}_3(t) \\
+2 t (-6 + 7 t) (-1 + 9 t) {\frac {d^3}{dt^3}}\mathcal{G}_3(t) + (-1 + t) t^2 (-1 + 9 t) {\frac {d^4}{dt^4}}\mathcal{G}_3(t) = 0. 
\end{multline*}
\end{prop}

Next, for $k = 0, 1, 2, 3$, we prove a uniform recurrence equation for the
sequence $\mathcal{S}_k$. It is given by a single formula with $k$ as a
parameter. Moreover, we show that $\mathcal{S}_0, \mathcal{S}_1,
\mathcal{S}_2, \mathcal{S}_3$ are identical to the sequences in the second family.

\begin{thm} \label{THM:quadrantrec}
For $k = 0, 1, 2, 3$, the $n$-th term $a(n)$ of the sequence $\mathcal{S}_{k}$ satisfies the following recurrence equation:
\begin{multline}
 (-3 + k)^2 (-2 + k) (6 + k) (1 + n) (2 + n) a(n) +  \\
-2 (-3 + k) (2 + n) (-60 + 8 k + 8 k^2 - 18 n + 3 k n + 2 k^2 n) a(n + 1) + \\  
( -342 - 174 k + 114 k^2 - 195 n - 70 k n + 54 k^2 n - 27 n^2 - 
 6 k n^2 + 6 k^2 n^2 )
a(n + 2) \\ + 
 2 
( 57 - 70 k + 16 n - 24 k n + n^2 - 2 k n^2)
a(n + 3) \\ 
 + (n + 7) (n + 8) a(n + 4) = 0. \label{EQ:uniformrec}
\end{multline} 
\end{thm}
\begin{proof}
By Lemma~\ref{LEM:binomialrep}, sequences $\mathcal{S}_2, \mathcal{S}_1,
\mathcal{S}_0$ are the first, second, and third inverse binomial transforms of $\mathcal{S}_3$,
respectively. Thus, it follows from Lemma~\ref{lem:btg} that the generating function of $\mathcal{S}_k$
is
\[
\mathcal{G}_k(t) = \frac{1}{1 + (3-k) t} \cdot \mathcal{G}_3 \left(\frac{t}{1+ (3 -k)t} \right) \ \ \text{ for } \ \ k = 0, 1, 2, 3,
\] 
where $\mathcal{G}_3(t)$ is the generating function of $\mathcal{S}_3$. Regarding $k$ as a
parameter in the above expression, we deduce the differential equation for
$\mathcal{G}_k(t)$ by using Proposition~\ref{LEM:marberg} and
closure properties of D-finite functions. By converting the differential
equation for $\mathcal{G}_k(t)$, we get the corresponding recurrence equation for
the sequence $a(n)$, which is exactly the recurrence equation in the claim.
\end{proof}

The argument of~\cite[Proposition 10]{MR2681853} gives rise to a proof for the recurrence relation of~\oeis{A151366} stated in the OEIS. For the Baxter sequence~\oeis{A001181}, the recurrence relation can be verified by utilizing its explicit formula in terms of
binomial coefficients and the creative telescoping method. The details for the verification can be found in~\cite{ElectronicYZ}. 
The recurrence equation for~\oeis{A216947} is proven in~\cite{Marberg2012}. 
As far as we know, the stated recurrence relation in the OEIS for~\oeis{A236408} is 
only conjectured. Next, we give a unified proof for all those recurrence relations. 

\begin{cor} \label{COR:identical}
The recurrence relations stated in the OEIS for the sequences in the second family specified in Figure~\ref{fig:bx} are true. 
Moreover, the sequences in the second
family are related by binomial transforms.
\end{cor}

\begin{proof}

By Theorem~\ref{THM:quadrantid}, we see that 
the quadrant sequences $\mathcal{S}_0, \mathcal{S}_1, \mathcal{S}_2$ and~$\mathcal{S}_3$ 
are identical to the sequences in the second family specified in Figure~\ref{fig:bx}. 
In~\eqref{EQ:uniformrec}, by setting $k$ to $0, 1, 2, 3$, we find that the
corresponding recurrence operators are left multiples~\cite[page
618]{Chen2016} of those of~\oeis{A216947}, \oeis{A001181}, \oeis{A236408},
and~\oeis{A151366} specified in the OEIS, respectively. 
To verify that they satisfies those recurrence relations, we just need to check finitely many
initial terms. The details of the verification can be found
in~\cite{ElectronicYZ}.
 Since $\mathcal{S}_k$'s are related by binomial transforms, so
are sequences in the second family.
\end{proof}

Closed formulae for these sequences can be obtained by the same methods as in Section~\ref{SUBSEC:closedformulae}.

Define the function 
\[
\setlength\arraycolsep{1pt}
H(x)  = 
{}_2 F_1\left(\begin{matrix}\frac{1}{3}& &\frac{2}{3}\\& 1 &
\end{matrix};\frac{27x^2}{(1-2x)^3}\right).
\]

Then  the ordinary generating function of $\mathcal{S}_2$  is equal to
\begin{multline*}
    \frac{(x+1)^2\,(1-8x)}{3x^2\,(1-2x)^2}\left(
H + \frac{12(1+20x-8x^2)(1-2x)}{(x+1)}    H'
    \right) \\ - \frac{(3x^2-x+1)}{(3x^2)}
\end{multline*}
The above formula is equivalent to the following formula for \oeis{
A001181} in the OEIS.

\begin{multline*}
\setlength\arraycolsep{1pt}
  -1 + \frac1{3x^2} \left[ (x-1) + (1-2x) {}_2 F_1\left(\begin{matrix}\frac{-2}{3}& &\frac{2}{3}\\& 1 &
\end{matrix};\frac{27x^2}{(1-2x)^3}\right) \right.\\
\left. - \frac{(8x^3-11x^2-x)}{(1-2x)^2}{}_2 F_1\left(\begin{matrix}\frac{1}{3}& &\frac{2}{3}\\& 2 &
\end{matrix};\frac{27x^2}{(1-2x)^3}\right) \right].
\end{multline*}

A closed formula for the generating function of $\mathcal{S}_k$ follows by substitution in Lemma~\ref{lem:btg}.

The asymptotics is stated in \cite{MR491652} and a detailed account is given in \cite[Theorem~1]{felsner}.

\end{document}